\documentclass{amsart}
\pdfoutput=1
\newif\ifprivate

\ifprivate
\usepackage[mark]{gitinfo2}
\renewcommand{\gitMark}{\jobname\,\textbullet{}\,\gitFirstTagDescribe\,\textbullet{}\,\gitAuthorName,\,\gitAuthorIsoDate}
\newcommand{\TODO}[1]%
{\par\fbox{\begin{minipage}{0.9\linewidth}\textbf{TODO:} #1\end{minipage}}\par}
\fi

\usepackage{a4wide}
\usepackage{amssymb}
\usepackage{hyperref}
\usepackage[utf8]{inputenc}
\usepackage{mathtools}

\newcommand{\bfr}{\mathbf{r}}

\newcommand{\bfmu}{\boldsymbol{\mu}}
\newcommand{\bfone}{\boldsymbol{1}}
\newcommand{\bfOmega}{\boldsymbol{\Omega}}
\newcommand{\bfs}{\mathbf{s}}
\newcommand{\bfa}{\mathbf{a}}
\newcommand{\bfQ}{\mathbf{Q}}
\newcommand{\bfk}{\mathbf{k}}
\newcommand{\bft}{\mathbf{t}}
\newcommand{\bfw}{\mathbf{w}}
\newcommand{\bfx}{\mathbf{x}}
\newcommand{\bfX}{\mathbf{X}}
\newcommand{\bfy}{\mathbf{y}}
\newcommand{\bfY}{\mathbf{Y}}
\newcommand{\bfz}{\mathbf{z}}
\newcommand{\bfZ}{\mathbf{Z}}
\newcommand{\bfzero}{\boldsymbol{0}}
\newcommand{\bfzeta}{\boldsymbol{\zeta}}

\newcommand{\C}{\mathbb{C}}

\newcommand{\calS}{\mathcal{S}}
\newcommand{\E}{\mathbb{E}}
\DeclareMathOperator{\grad}{grad}
\newcommand{\ip}[2]{\langle #1, #2\rangle}
\renewcommand{\MR}[1]{}
\renewcommand{\P}{\mathbb{P}}
\newcommand{\R}{\mathbb{R}}

\newcommand{\sphi}{\sqrt{\phi_{n}}}
\newcommand{\sphii}{\phi_n^{-1/2}}
\newcommand{\stirlingpartition}[2]{\genfrac{\{}{\}}{0pt}{}{#1}{#2}}

\DeclareMathOperator{\Cov}{Cov}
\newcommand{\cupdot}{\mathbin{\mathaccent\cdot\cup}}

\DeclarePairedDelimiter{\abs}{\lvert}{\rvert}
\DeclarePairedDelimiter{\norm}{\lVert}{\rVert}
\DeclarePairedDelimiter{\iverson}{[}{]}

\newtheorem{theorem}{Theorem}
\newtheorem*{theoremnonumber}{Theorem}
\newtheorem{lemma}{Lemma}[section]
\newtheorem{proposition}[lemma]{Proposition}
\newtheorem{corollary}[lemma]{Corollary}

\theoremstyle{definition}
\newtheorem{definition}[lemma]{Definition}

\theoremstyle{remark}

\numberwithin{equation}{section}

\title{Higher Dimensional Quasi-Power Theorem and Berry--Esseen Inequality}
\author{Clemens Heuberger}
\address{Institut f\"ur Mathematik, Alpen-Adria-Universit\"at Klagenfurt, Austria}
\email{clemens.heuberger@aau.at}
\thanks{The authors are supported by the Austrian Science Fund (FWF):
  P~24644-N26.\\ This is the full version of the extended abstract \cite{Heuberger-Kropf:2016:quasi-power-arxiv}.}
\author{Sara Kropf}
\address{Institut f\"ur Mathematik, Alpen-Adria-Universit\"at
  Klagenfurt, Austria and Institute of Statistical Science, Academia
  Sinica, Taipei, Taiwan}
\email{sara.kropf@aau.at, sarakropf@stat.sinica.edu.tw}

\begin{document}

\begin{abstract}Hwang's quasi-power theorem asserts that a sequence of
  random variables whose moment generating functions are approximately given by
  powers of some analytic function is asymptotically normally distributed.
  This theorem is generalised to higher dimensional random
  variables. To obtain this result, a higher dimensional analogue of the Berry--Esseen
  inequality is proved, generalising a two-dimensional version by
  Sadikova.
\end{abstract}

\maketitle

\section{Introduction}
Asymptotic normality is a frequently occurring phenomenon in combinatorics, the
classical central limit theorem being the very first example. The first step in
the proof is the observation that the
moment generating function of the sum of $n$ identically independently distributed
random variables is the $n$-th power of the moment generating function of the
distribution underlying the summands. As similar moment generating functions
occur in many examples in combinatorics, a general theorem to prove asymptotic
normality is desirable. Such a theorem was proved by
Hwang~\cite{Hwang:1998}, usually called the ``quasi-power theorem''.

\begin{theoremnonumber}[Hwang~\cite{Hwang:1998}]
  Let $\{\Omega_n\}_{n\ge 1}$ be a sequence of  integral random
  variables. Suppose that the moment generating function satisfies the asymptotic expression
  \begin{equation}\label{eq:moments-1d}
    M_n(s):=\E(e^{\Omega_ns})=e^{W_n(s)}(1+O(\kappa_n^{-1})),
  \end{equation}
  the $O$-term being uniform for $\abs{s}\le \tau$, $s\in\C$,
  $\tau>0$, where
  \begin{enumerate}
  \item $W_n(s)=u(s)\phi_{n}+v(s)$, with $u(s)$ and $v(s)$
    analytic for $\abs{s}\le \tau$ and independent of $n$; and $u''(0)\neq 0$;
  \item $\lim_{n\to\infty}\phi_{n}=\infty$;
  \item $\lim_{n\to\infty}\kappa_n=\infty$.
  \end{enumerate}

  Then the distribution of $\Omega_n$ is asymptotically normal, i.e.,
  \begin{equation*}
    \sup_{x\in\R}\bigg\vert\P\bigg(\frac{\Omega_n- u'(0)\phi_{n}}{\sqrt{u''(0)\phi_{n}}} <
      x\bigg)-
    \Phi(x)\bigg\vert=O\bigg(\frac{1}{\sqrt{\phi_{n}}}+\frac{1}{\kappa_n}\bigg),
  \end{equation*}
  where $\Phi$ denotes the standard normal distribution
  \begin{equation*}
    \Phi(x)=\frac{1}{\sqrt{2\pi}}\int_{-\infty}^{x}\exp\Big(-\frac12
      y^2\Big)\,dy.
  \end{equation*}
\end{theoremnonumber}

See Hwang's article~\cite{Hwang:1998} as well as
Flajolet-Sedgewick~\cite[Sec.~IX.5]{Flajolet-Sedgewick:ta:analy} for many
applications of this theorem.
A generalisation of the quasi-power theorem to dimension~$2$ has been provided
 in \cite{Heuberger:2007:quasi-power}. It has been used in \cite{Heuberger-Prodinger:2006:analy-alter},
 \cite{Heuberger-Prodinger:2007:hammin-weigh},
 \cite{Eagle-Gao-Omar-Panario:2008:distr:short},
 \cite{Heuberger-Kropf-Wagner:2014:combin-charac} and
 \cite{Kropf:2015:varian-and}. In~\cite[Thm.~2.22]{Drmota:2009:random},
 an $m$-dimensional version of the quasi-power theorem is stated
 without speed of convergence. Also
 in~\cite{Bender-Richmond:1983:centr}, such an $m$-dimensional theorem
 without speed of convergence is proved. There, several multidimensional applications
 are given, too. 

In contrast to many results about the speed of convergence in classical probability
 theory (see, e.g.,~\cite{Gut:2005:probab}), the sequence of
 random variables is not assumed to be independent. The only assumption is that the
 moment generating function behaves asymptotically like a large
 power. This mirrors the fact that the moment generating function of
 the sum of independent, identically distributed random variables is exactly a
 large power. The advantage is that the asymptotic expression~\eqref{eq:moments-1d} arises naturally in combinatorics by
 using techniques such as singularity analysis or saddle point
 approximation (see \cite{Flajolet-Sedgewick:ta:analy}).  

The purpose of this article is to generalise the quasi-power theorem
including the speed of convergence to
arbitrary dimension $m$. We first state this main result in
Theorem~\ref{th:quasi-power-dD} in this section. In
Section~\ref{sec:Berry--Esseen}, a new Berry--Esseen inequality
(Theorem~\ref{theorem:Berry-Esseen-dimension-m}) is presented, which we
use to prove the $m$-dimensional quasi-power theorem. In
Section~\ref{sec:exampl-mult-centr}, we give some applications of the
multidimensional quasi-power theorem. The combinatorial idea behind the formulation
of the Berry--Esseen inequality is discussed in
Section~\ref{sec:operator-Lambda}. Our Berry--Esseen bound is proved in
Section~\ref{sec:proof-Berry--Esseen}. The final
Section~\ref{sec:proof-quasi-power-theorem} is then devoted to the proof of the
quasi-power theorem.

We use the following conventions: vectors are denoted by
boldface letters such as $\bfs$, their components are then denoted by
regular letters with indices such as $s_j$. For a vector $\bfs$, $\|\bfs\|$ denotes the
maximum norm $\max\{\abs{s_j}\}$.
All implicit constants of
$O$-terms may depend on the dimension $m$ as well as on $\tau$ which is introduced in
Theorem~\ref{th:quasi-power-dD}.

Our first main result is the following $m$-dimensional version of Hwang's theorem.

\begin{theorem}\label{th:quasi-power-dD}
  Let $\{\bfOmega_n\}_{n\ge 1}$ be a sequence of $m$-dimensional real random
  vectors. Suppose that the moment generating function satisfies the asymptotic expression
  \begin{equation}\label{eq:moment-asymp}
    M_n(\bfs):=\E(e^{\langle \bfOmega_n,\bfs\rangle})=e^{W_n(\bfs)}(1+O(\kappa_n^{-1})),
  \end{equation}
  the $O$-term being uniform for $\norm{\bfs}\le \tau$, $\bfs\in\C^m$,
  $\tau>0$, where
  \begin{enumerate}
  \item $W_n(\bfs)=u(\bfs)\phi_{n}+v(\bfs)$, with $u(\bfs)$ and $v(\bfs)$
    analytic for $\norm{\bfs}\le \tau$ and independent of $n$; and the Hessian
    $H_u(\bfzero)$ of $u$ at the origin is non-singular;
  \item $\lim_{n\to\infty}\phi_{n}=\infty$;
  \item $\lim_{n\to\infty}\kappa_n=\infty$.
  \end{enumerate}

  Then, the distribution of $\bfOmega_n$ is
  asymptotically normal with speed of convergence $O(\sphii)$, i.e.,
  \begin{equation}\label{eq:quasi-power-result}
    \sup_{\bfx\in\R^{m}}\bigg\vert\P\bigg(\frac{\bfOmega_n-\grad u (\bfzero)\phi_{n}}{\sqrt{\phi_{n}}} \le
      \bfx\bigg)-
    \Phi_{H_u(\bfzero)}(\bfx)\bigg\vert=O\left(\frac{1}{\sqrt{\phi_{n}}}\right),
  \end{equation}
  where $\Phi_{\Sigma}$ denotes the distribution function
  of the non-degenerate $m$-dimensional normal distribution
  with mean $\bfzero$ and variance-covariance
  matrix $\Sigma$, i.e.,
  \begin{equation*}
    \Phi_\Sigma(\bfx)=\frac{1}{(2\pi)^{m/2}\sqrt{\det \Sigma}}\int_{\bfy\le \bfx}\exp\Big(-\frac12
      \bfy^\top \Sigma^{-1} \bfy \Big)\,d\bfy,
  \end{equation*}
  where $\bfy\le \bfx$ means $y_\ell\le x_\ell$ for $1\le \ell\le m$.

  If $H_{u}(\bfzero)$ is singular, the random variables
  \begin{equation*}
    \frac{\bfOmega_{n}-\grad u(\bfzero)\phi_{n}}{\sqrt{\phi_{n}}}
  \end{equation*}
  converge in distribution to a
  degenerate normal distribution with mean $\bfzero$ and variance-covariance
  matrix $H_{u}(\bfzero)$.
\end{theorem}

Note that in the case of the singular $H_{u}(\bfzero)$, a uniform speed of
convergence cannot be guaranteed.
 To see this,
  consider the (constant) sequence of random variables $\Omega_{n}$
  which takes values $\pm1$ each with probability $1/2$. Then the
  moment generating function is $(e^{t}+e^{-t})/2$,
which is of the form \eqref{eq:moment-asymp} with $\phi_{n}=n$,
$u(s)=0$, $v(s)=\log (e^{t}+e^{-t})/2$ and $\kappa_{n}$ arbitrary. However, the distribution function of $\Omega_{n}/\sqrt{n}$ is
  given by
  \begin{equation*}
    \mathbb{P}\biggl(\frac{\Omega_{n}}{\sqrt{n}}\le x\biggr)=
    \begin{cases}
      0& \text{if }x<-1/\sqrt{n},\\
      1/2& \text{if }-1/\sqrt{n}\le x<1/\sqrt{n},\\
      1& \text{if }1/\sqrt{n}\le x,
    \end{cases}
  \end{equation*}
  which does not converge uniformly.

In contrast to the original quasi-power theorem, the error term in our result does not
contain the summand $O(1/\kappa_n)$. In fact, this summand could also be
omitted in the original proof of the quasi-power theorem by using a better
estimate for the error $E_n(s)=M_n(s)e^{-W_{n}(s)}-1$, cf.\ the proof of our
Lemma~\ref{lemma:characteristic-function-single-bound}.

The order of the error is
optimal (without further assumptions on the random variables), as it is
the case for the one-dimensional Berry-Esseen inequality. See,
for example, the approximation of a binomial distribution by the
normal distribution \cite[\S~1.2]{Petrov:2000:class-type}.

The proof of Theorem~\ref{th:quasi-power-dD} relies on an
$m$-dimensional Berry--Esseen inequality
(Theorem~\ref{theorem:Berry-Esseen-dimension-m}). It is a generalisation of
Sadikova's result~\cite{Sadikova:1966:esseen, Sadikova:1966:esseen:englisch} in
dimension $2$. The main challenge is to provide a version which leads to
bounded integrands around the origin, but still allows to use excellent bounds for the
tails of the characteristic functions. To achieve this, linear combinations
involving all partitions of the set $\{1,\ldots, m\}$ are used.

Note that there are several generalisations of the one-dimensional 
Berry--Esseen inequality \cite{Berry:1941:gauss,Esseen:1945:fourier}
to arbitrary dimension, see, e.g., Gamkrelidze~\cite{Gamkrelidze:1977,
  Gamkrelidze:1977:englisch} and Prakasa
Rao~\cite{Rao:2002:anoth-esseen}. 
 However, using these results would lead to
a less precise error term in~\eqref{eq:quasi-power-result}, see the end of
Section~\ref{sec:Berry--Esseen} for more details. For that reason we generalise Sadikova's
result, which was already successfully used by the first author
in~\cite{Heuberger:2007:quasi-power} to prove a $2$-dimensional
quasi-power theorem. Also note
 that our theorem can deal with discrete random variables, too, in contrast
 to \cite{Roussas:2001:esseen}, where density functions are considered.

For the sake of completeness, we also state the following result
about the moments of $\bfOmega_{n}$.
\begin{proposition}\label{proposition:moments}
  The cross-moments of $\bfOmega_{n}$ satisfy
  \begin{equation*}
    \frac{1}{\prod_{\ell=1}^{m}k_{\ell}!}\mathbb E\bigg(\prod_{\ell=1}^{m}\Omega_{n,\ell}^{k_{\ell}}\bigg)=p_{\bfk}(\phi_{n})+O\big(\kappa_{n}^{-1}\phi_{n}^{k_{1}+\cdots+k_{m}}\big),
  \end{equation*}
for $k_{\ell}$ nonnegative integers, where $p_{\bfk}$ is a polynomial of degree
$\sum_{\ell=1}^{m}k_{\ell}$ defined by
\begin{equation*}
  p_{\bfk}(X)=[s_{1}^{k_{1}}\cdots s_{m}^{k_{m}}]e^{u(\bfs)X+v(\bfs)}.
\end{equation*}

In particular, the mean and the variance-covariance matrix are
\begin{align*}
  \mathbb E(\bfOmega_{n})&=\grad u(\bfzero)\phi_{n}+\grad
  v(\bfzero)+O(\kappa_{n}^{-1}),\\
  \Cov(\bfOmega_{n})&=H_{u}(\bfzero)\phi_{n}+H_{v}(\bfzero)+O(\kappa_{n}^{-1}),
\end{align*}
respectively.
\end{proposition}

\section{A Berry--Esseen Inequality}\label{sec:Berry--Esseen}

This section is devoted to a generalisation of Sadikova's Berry--Esseen
inequality~\cite{Sadikova:1966:esseen, Sadikova:1966:esseen:englisch} in
dimension 2 to dimension $m$. Before stating the theorem, we introduce our notation.

Let $L=\{1,\ldots, m\}$. For $K\subseteq L$, we write
$\bfs_K=(s_k)_{k\in K}$ for the projection of $\bfs\in\C^L$ to $\C^K$.
For $J\subseteq K\subseteq L$, let
$\chi_{J,K}\colon \C^{J}\to\C^{K}$, $(s_{j})_{j\in J}\mapsto
(s_{k}\iverson{k\in J})_{k\in K}$ be an injection from $\C^{J}$ into
$\C^{K}$. Similarly, let $\psi_{J,K}\colon \C^{K}\to\C^{K}$,
$(s_{k})_{k\in K}\mapsto (s_{k}\iverson{k\in J})_{k\in K}$ be the
projection which sets all coordinates corresponding to $K\setminus J$ to
$0$.

We denote the
  set of all partitions of $K$ by $\Pi_K$. We consider a partition as a set
  $\alpha=\{J_{1},\ldots,J_{k}\}$. Thus $\abs{\alpha}$ denotes the number of parts of the partition
$\alpha$. Furthermore, $J\in\alpha$ means that $J$ is a part of the
partition $\alpha$.

Now, we can define an operator which we later use to state our Berry--Esseen
inequality. The motivation behind this definition is explained at the end of
this section.

\begin{definition}\label{definition:Lambda-K}
  Let $K\subseteq L$ and $h\colon \C^K\to \C$. We define the non-linear operator
  \begin{equation*}
    \Lambda_K(h):=\sum_{\alpha\in\Pi_K}\mu_\alpha \prod_{J\in
      \alpha}h\circ \psi_{J, K}
  \end{equation*}
  where
  \begin{equation*}
    \mu_\alpha = (-1)^{\abs{\alpha}-1}(\abs{\alpha}-1)!\,.
  \end{equation*}

We denote $\Lambda_{L}$ briefly by $\Lambda$.
\end{definition}

For any random variable $\bfZ$, we denote its cumulative distribution function
by $F_\bfZ$, its density function by $f_\bfZ$ (if it exists) and its
characteristic function by $\varphi_\bfZ$.

With these definitions, we are able to state our second main result, an
$m$-dimensional version of the Berry--Esseen inequality.

\begin{theorem}\label{theorem:Berry-Esseen-dimension-m}
Let $m\ge 1$ and $\bfX$ and $\bfY$ be $m$-dimensional random variables. Assume
that $F_\bfY$ is differentiable.

Let
\begin{align*}
  A_j&=\sup_{\bfy \in\R^m}\frac{\partial F_\bfY(\bfy)}{\partial y_j},\\
  B_j&=\sum_{k=1}^{j} \stirlingpartition{j}{k} k!\ ,\\
  C_1&=\sqrt[3]{\frac{32}{\pi\bigl(1-\bigl(\frac{3}{4}\bigr)^{1/m}\bigr)}},\\
  C_2&=\frac{12}{\pi}
\end{align*}
for $1\le
  j\le m$ where $\stirlingpartition{j}{k}$ denotes a Stirling partition number
(Stirling number of the second kind).

Let $T>0$ be fixed. Then
\begin{equation}\label{eq:Berry-Esseen}
  \begin{aligned}
    \sup_{\bfz\in\R^m}\abs{F_{\bfX}(\bfz)-F_{\bfY}(\bfz)}&\le
    \frac{2}{(2\pi)^m} \int_{\norm{\bft}\le
      T}\abs[\bigg]{\frac{\Lambda(\varphi_{\bfX})(\bft)-\Lambda(\varphi_{\bfY})(\bft)}{\prod_{\ell\in
          L} t_\ell}}\,d\bft \\
    &\qquad+ 2\sum_{\emptyset\neq J\subsetneq
      L}B_{m-\abs{J}}\sup_{\bfz_J\in\R^J}\abs[\big]{F_{\bfX_{J}}(\bfz_J)-F_{\bfY_{J}}(\bfz_J)}
    \\
    &\qquad +\frac{2\sum_{j=1}^m A_j}{T}(C_1+C_2).
  \end{aligned}
\end{equation}
Existence of $\E(\bfX)$ and $\E(\bfY)$ is sufficient for the finiteness
of the integral in \eqref{eq:Berry-Esseen}.
\end{theorem}
Let us give two remarks on the distribution functions occurring in this theorem:
The distribution function $F_\bfY$ is non-decreasing in every
variable, thus $A_j>0$ for all $j$.
Furthermore, our general notations imply that $F_{\bfX_J}$ is a marginal
distribution of $\bfX$.

The numbers $B_j$ are known as ``Fubini numbers'' or ``ordered Bell numbers''.
They form the sequence \href{http://oeis.org/A000670}{A000670} in \cite{OEIS:2015}.

Recursive application of \eqref{eq:Berry-Esseen} leads to the following
corollary, where we no longer explicitly state the constants depending on the
dimension.

\begin{corollary}\label{corollary:Berry-Esseen}
  Let $m\ge 1$ and $\bfX$ and $\bfY$ be $m$-dimensional random
  variables. Assume
that $F_\bfY$ is differentiable and
  let
  \begin{equation*}
    A_j=\sup_{\bfy \in\R^m}\frac{\partial F_\bfY(\bfy)}{\partial y_j}, \qquad 1\le
    j\le m.
  \end{equation*}

  Then
  \begin{multline}\label{eq:Berry-Esseen-recursive}
      \sup_{\bfz\in\R^m}\abs{F_{\bfX}(\bfz)-F_{\bfY}(\bfz)}\\=
      O\biggl(\sum_{\emptyset \neq K\subseteq L}\int_{\norm{\bft_K}\le
        T}\abs[\bigg]{\frac{\Lambda_K(\varphi_{\bfX}\circ\chi_{K, L})(\bft_K)-\Lambda_K(\varphi_{\bfY}\circ\chi_{K, L})(\bft_K)}{\prod_{k\in
            K} t_k}}\,d\bft_K + \frac{\sum_{j=1}^m A_j}{T}\biggr)
  \end{multline}
  where the $O$-constants only depend on the dimension $m$.

  Existence of $\E(\bfX)$ and $\E(\bfY)$ is sufficient for the finiteness
  of the integrals in \eqref{eq:Berry-Esseen-recursive}.
\end{corollary}

In order to explain the choice of the operator $\Lambda$, we first state it in
dimension $2$:
\begin{equation}\label{eq:Sadikova-simple}
  \Lambda(h)(s_1, s_2) = h(s_1, s_2) - h(s_1, 0)h(0, s_2).
\end{equation}
This coincides with Sadikova's definition. This also shows that our
operator is non-linear as, e.g., $\Lambda(s_{1}+s_{2})(s_{1},s_{2})\neq\Lambda(s_{1})(s_{1},s_{2})+\Lambda(s_{2})(s_{1},s_{2})$.

In Theorem~\ref{theorem:Berry-Esseen-dimension-m}, we apply $\Lambda$ to
characteristic functions; so we may restrict our attention to functions $h$
with $h(\bfzero)=1$. From~\eqref{eq:Sadikova-simple}, we see that
$\Lambda(h)(s_1, 0) = \Lambda(h)(0, s_2)=0$, so that $\Lambda(h)(s_1,
s_2)/(s_1s_2)$ is bounded around the origin. This is essential for the boundedness
of the integral in Theorem~\ref{theorem:Berry-Esseen-dimension-m}. In general,
this property will be guaranteed by our particular choice of coefficients. It
is no coincidence that for $\alpha\in \Pi_L$, the coefficient $\mu_\alpha$
equals the value  $\mu(\alpha, \{L\})$ of the Möbius function in the lattice of partitions:
Weisner's theorem (see Stanley~\cite[Corollary~3.9.3]{Stanley:2012:enumer_1}) is crucial in the proof that $\Lambda(h)(\bfs)/(s_1\ldots
s_m)$ is bounded around the origin (see the proof of Lemma~\ref{lemma:Lambda-property}).

The second property is that our proof of the quasi-power theorem needs
estimates for the tails of the integral in Theorem~\ref{theorem:Berry-Esseen-dimension-m}. These estimates have to be
exponentially small in every variable, which means that every variable has to
occur in every summand. This is trivially fulfilled as every summand
in the definition of $\Lambda$ is
formulated in terms of a partition.

Note that Gamkrelidze~\cite{Gamkrelidze:1977:englisch}  (and also
Prakasa Rao~\cite{Rao:2002:anoth-esseen}) use a linear operator $L$ mapping $h$ to
\begin{equation}\label{eq:gamkrelidze}
  (s_1, s_2) \mapsto h(s_1, s_2) - h(s_1, 0) - h(0, s_2).
\end{equation}
When taking the difference of two characteristic functions, we may assume that
$h(0, 0)=0$ so that the first crucial property as defined above still
holds. However, the tails are no longer exponentially small in every variable: the last summand $h(0,s_{2})$ in~\eqref{eq:gamkrelidze} is not exponentially small in
$s_{1}$ because it is independent of $s_{1}$ and nonzero in
general. However, the first two summands are exponentially small in
$s_{1}$ by our assumption~\eqref{eq:moment-asymp}.

For that reason, using the Berry--Esseen inequality by
Gamkrelidze~\cite{Gamkrelidze:1977:englisch} to prove a quasi-power theorem leads to a less precise error term 
$O(\phi_{n}^{-1/2}\log^{m-1}\phi_n)$
in~\eqref{eq:quasi-power-result}. It can be shown that the less precise
error term necessarily appears when using Gamkrelidze's result by considering
the example of $\bfOmega_n$ being the $2$-dimensional vector consisting of a normal distribution
with mean $-1$ and variance $n$ and a normal distribution with mean $0$ and
variance $n$. This is a consequence of the linearity of the
operator $L$ in Gamkrelidze's result.

\section{Examples of Multidimensional Central Limit
  Theorems}\label{sec:exampl-mult-centr}
In this section, we give two examples from combinatorics where we can
apply Theorem~\ref{th:quasi-power-dD}. Asymptotic normality was
already shown in earlier publications
\cite{Drmota:1997:system-funct-equat,Bender-Richmond:1983:centr}, but
we additionally provide an estimate for the
speed of convergence.
\subsection{Context-Free Languages}
Consider the following example of a context-free grammar $G$ with
non-terminal symbols $S$ and $T$, terminal symbols $\{a,b,c\}$,
starting symbol $S$ and the rules 
\begin{equation*}
  P=\{S\to aSbS,\, S\to bT,\, T\to bS,\, T\to cT,\, T\to a\}.
\end{equation*}
The corresponding context-free language $L(G)$ consists of all words which
can be generated starting with $S$ using the rules in $P$ to replace
all non-terminal symbols. For example, $abcabababba\in L(G)$ because it can be derived as
\begin{equation*}
  S\to aSbS
  \to abTbaSbS
\to abcTbabTbbT
\to abcabababba.
\end{equation*}

Let $\P(\bfOmega_{n}=\bfx)$ be the probability that a word of length $n$
in $L(G)$ consists of $x_{1}$ and $x_{2}$ terminal symbols
$a$ and $b$, respectively. Thus there are $n-x_{1}-x_{2}$ terminal
symbols $c$. For simplicity, this random variable is only
$2$-dimensional. But it can be easily
extended to higher dimensions.

Following
Drmota~\cite[Sec.~3.2]{Drmota:1997:system-funct-equat}, we obtain that the
moment generating function is
\begin{equation*}
  \E(e^{\langle
    \bfOmega_n,\bfs\rangle})=\frac{y_{n}(e^{\bfs})}{y_{n}(\bfone)}
\end{equation*}
with $y_{n}(\boldsymbol{z})$ defined
in \cite{Drmota:1997:system-funct-equat}. Using
\cite[Equ.~(4.9)]{Drmota:1997:system-funct-equat}, this moment
generating function has an asymptotic expansion as in \eqref{eq:moment-asymp} with $\phi_{n}=n$. Thus
$\bfOmega_{n}$ is asymptotically normally distributed after
standardisation (as was shown in~\cite{Drmota:1997:system-funct-equat})
and additionally the speed of convergence is $O(n^{-1/2})$.

Other context-free languages can be analysed in the same way, either
by directly using the results in \cite{Drmota:1997:system-funct-equat} (if the
underlying system is strongly connected) or by similar methods. This
has applications, for example, in genetics (see~\cite{Poznanovic-Heitsch:2014:asymp-rna}).

\subsection{Dissections of Labelled Convex Polygons}
Let $S_{1}\cupdot\cdots\cupdot S_{t+1}=\{3,4,\ldots\}$ be a
partition. We dissect a labelled convex $n$-gon into smaller convex polygons by
choosing some non-intersecting diagonals. Each small polygon should be a
$k$-gon with $k\not\in S_{t+1}$. Define $a_{n}(\bfr)$ to be the number
of dissections of an $n$-gon such that it consists of exactly $r_{i}$ small
polygons whose number of vertices is in $S_{i}$, for $i=1$, \dots,
$t$. For convenience, we use $a_{2}(\bfr)=[\bfr=\bfzero]$. Asymptotic normality was proved in
\cite[Sec.~3]{Bender-Richmond:1983:centr}, see also
\cite[Ex.~7.1]{Bender:1974:asymp-method-enumer} for a one-dimensional
version. We additionally provide an estimate for the speed of convergence.

Let
\begin{equation*}
  f(z,\bfx)=\sum_{\substack{n\geq2\\ \bfr\geq 0}}a_{n}(\bfr)\bfx^{\bfr}z^{n-1}.
\end{equation*}
Then choosing a $k$-gon with $k\in S_{1}\cupdot\cdots\cupdot S_{t}$
and gluing dissected polygons to $k-1$ of its sides translates into
the equation
\begin{equation*}
  f=z+\sum_{i=1}^{t}x_{i}\sum_{k\in S_{i}}f^{k-1}.
\end{equation*}
Following \cite{Bender:1974:asymp-method-enumer}, this equation can be used to obtain an
asymptotic expression for the moment generating function as in
\eqref{eq:moment-asymp} with $\phi_{n}=n$. The asymptotic normal
distribution follows after suitable standardisation with speed of
convergence $O(n^{-1/2})$.

\section{Combinatorial Background of the Operator \texorpdfstring{$\Lambda$}{Λ}}\label{sec:operator-Lambda}
Before we start with the proof of Theorem~\ref{theorem:Berry-Esseen-dimension-m}, we
state and prove the property of our operator $\Lambda$
 which motivates its Definition~\ref{definition:Lambda-K}.

\begin{lemma}\label{lemma:Lambda-property}
  Let $K\subsetneq L$ and $h\colon \C^L\to \C$ with $h(\bfzero)=1$. Then
  \begin{equation*}
    \Lambda(h)\circ \psi_{K, L}=0.
  \end{equation*}
\end{lemma}

Before actually proving the lemma, we recall some of the theory about the Möbius
function of a partially ordered set (poset), see also
Stanley~\cite[Section~3.7]{Stanley:2012:enumer_1}.

By the following definition, $\Pi_L$, the set of all partitions of
$L$, is a poset:
As usual, a partition $\alpha\in \Pi_L$ is said to be a refinement of a partition
$\alpha'\in\Pi_L$ if
\begin{equation*}
  \forall J\in\alpha\colon \exists J'\in\alpha'\colon J\subseteq J'.
\end{equation*}
In this case, we write $\alpha\le \alpha'$. This defines a partial order on
$\Pi_L$.

The Möbius function on $\Pi_L$ is denoted by $\mu$: for $\alpha<
\alpha'$, we set $\mu(\alpha', \alpha')=1$ and
\begin{equation*}
  \mu(\alpha, \alpha')=
    -\sum_{\substack{\beta\in \Pi_L\\\alpha<\beta\le\alpha'}} \mu(\beta,
    \alpha').
\end{equation*}
For $\alpha$, $\alpha'\in\Pi_L$, the infimum $\alpha\land \alpha'$ of $\alpha$
and $\alpha'$ is given by
\begin{equation*}
  \{ J\cap J'\colon J\in\alpha, J'\in\alpha', J\cap J'\neq \emptyset\}.
\end{equation*}
In fact, $\Pi_L$ is a lattice (cf.\
Stanley~\cite[Example~3.10.4]{Stanley:2012:enumer_1}).
The greatest element is $\{L\}$.

For $\alpha\in\Pi_L$, we have
\begin{equation*}
  \mu(\alpha, \{L\})=(-1)^{\abs{\alpha}-1}(\abs{\alpha}-1)!=\mu_\alpha,
\end{equation*}
where $\abs{\alpha}$ denotes the number of parts of the partition,
see Stanley~\cite[(3.37)]{Stanley:2012:enumer_1}. In particular, we may
rewrite the definition of $\Lambda$ (Definition~\ref{definition:Lambda-K}) as
\begin{equation}\label{eq:Lambda_K_definition-2}
  \Lambda(h):=\sum_{\alpha\in\Pi_L}\mu(\alpha, \{L\}) \prod_{J\in
      \alpha}h\circ \psi_{J, L}.
\end{equation}

For any $\gamma$, $\beta\in\Pi_L$ with $\gamma\le \beta< \{L\}$,
Weisner's theorem (see Stanley~\cite[Corollary~3.9.3]{Stanley:2012:enumer_1})
applied to the interval $[\gamma, \{L\}]$ asserts that
\begin{equation}\label{eq:Weisner}
  \sum_{\substack{\alpha\in\Pi_L\\ \alpha\land \beta=\gamma}}\mu(\alpha, \{L\})=0.
\end{equation}

We now turn to the actual proof of the lemma. 
\begin{proof}[Proof of Lemma~\ref{lemma:Lambda-property}.]
Consider the partition
$\beta=\{K\} \cup
\{\{k\}\colon k\in L\setminus K\}$ of $L$, i.e., $\beta$ consists of $K$ as one
part and a collection of singletons. As $K\neq L$, we have $\beta < \{L\}$.

By definition of $\psi$, we have
$\psi_{J,L}\circ \psi_{K,L}=\psi_{J\cap K, L}$ for $J$, $K\subseteq L$. If
$\alpha\in \Pi_L$, then
\begin{equation*}
  \prod_{J\in \alpha} h\circ \psi_{J\cap K, L}=
  \prod_{\substack{J\in
      \alpha\land\beta\\
    J\subseteq K}} h\circ \psi_{J, L}
\end{equation*}
because parts $J\in\alpha$ with $J\cap K=\emptyset$ contribute $h(\bfzero)=1$.
Therefore, collecting the sum \eqref{eq:Lambda_K_definition-2} according to
$\alpha\land\beta$ yields
\begin{equation*}
  \Lambda(h)\circ\psi_{K, L}:=\sum_{\alpha\in\Pi_L}\mu(\alpha, \{L\}) \prod_{J\in
      \alpha}h\circ \psi_{J\cap K, L}=
    \sum_{\gamma\in\Pi_L}\prod_{\substack{J\in
      \gamma\\
    J\subseteq K}} h\circ \psi_{J, L}
\sum_{\substack{\alpha\in\Pi_L\\\alpha\land\beta=\gamma}}\mu(\alpha, \{L\}).
\end{equation*}
As $\gamma\le\beta<\{L\}$, the inner sum vanishes by~\eqref{eq:Weisner}.
\end{proof}

\section{Proof of the Berry--Esseen Inequality}\label{sec:proof-Berry--Esseen}
This section is devoted to the proof of our Berry--Esseen inequality,
Theorem~\ref{theorem:Berry-Esseen-dimension-m}. It is a generalisation of
Sadikova's proof.

We start with an auxiliary one-dimensional random variable.
\begin{lemma}\label{lemma:P}
  Let $P$ be the one-dimensional random variable with probability density
  function
  \begin{equation*}
    f_P(z)=\frac{3}{8\pi}\Bigl(\frac{\sin(z/4)}{z/4}\Bigr)^4.
  \end{equation*}

  Then its characteristic function is
  \begin{equation}\label{eq:characteristic-function-P}
    \varphi_P(t)=
    \begin{cases}
      1-6t^2+6\abs{t}^3&\text{if $0\le \abs{t}\le 1/2$},\\
      2(1-\abs{t})^3&\text{if $1/2\le \abs{t}\le 1$},\\
      0&\text{if $1\le \abs{t}$}
    \end{cases}
  \end{equation}
  and
  \begin{align}
    \E(P^2)&=12,\notag\\
    \E(\abs{P})&\le C_2.\label{eq:expectation-P-absolute-value}
  \end{align}

  Let $\lambda$ be the unique positive number such that
  \begin{equation*}
    \P(P\le \lambda) = \P(P\ge -\lambda) = \Bigl(\frac{3}{4}\Bigr)^{1/m}.
  \end{equation*}
  Then
  \begin{equation}\label{eq:lambda-estimate}
    \lambda\le C_1.
  \end{equation}
\end{lemma}
\begin{proof}
  The characteristic function~\eqref{eq:characteristic-function-P} is
  mentioned in \cite[Section~39]{Gnedenko-Kolmogorov:1954:limit}; it is computed
  by standard methods.

  Differentiating $\varphi_P$ twice, we see that the second moment is $12$. To
  prove \eqref{eq:expectation-P-absolute-value}, we rewrite
  $\E(\abs{P})$ as
  \begin{equation*}
    \E(\abs{P})
    =\frac{12}{\pi} \int_0^1  \frac{\sin^4 z}{z^3}\, dz
    +\frac{12}{\pi} \int_1^\infty  \frac{\sin^4 z}{z^3}\, dz.
  \end{equation*}
  We use the estimates $\sin z\le z$ and $\abs{\sin z}\le 1$ on the intervals
  $[0, 1]$ and $[1, \infty)$, respectively. Thus
  \begin{equation*}
    \E(\abs{P})\le \frac{12}{\pi}\Bigl(\frac12 + \frac12\Bigr)=\frac{12}{\pi}.
  \end{equation*}
  
  To obtain a bound for $\lambda$, we follow
  Gamkrelidze~\cite{Gamkrelidze:1977:englisch}: we estimate the tail using $\abs{\sin^4(z)}\le 1$
  and get
    \begin{equation*}
      1-\Bigl(\frac{3}{4}\Bigr)^{1/m} =\frac{3}{8\pi}\int_{\lambda}^\infty
      \Bigl(\frac{\sin(z/4)}{z/4}\Bigr)^4\, dz\le
      \frac{3}{2\pi}\int_{\lambda/4}^\infty\Bigl(\frac{1}{z}\Bigr)^4\,dz
      =
      \frac3{2\pi}\Bigl(-\frac13\Bigr)\frac{1}{z^3}\Bigr\rvert_{z=\lambda/4}^\infty
      =\frac{32}{\pi\lambda^3}.
    \end{equation*}
    This results in  \eqref{eq:lambda-estimate}.
\end{proof}

In the next step, we consider tuples of random variables distributed as $P$.
They will be used to ensure smoothness. We write $\bfone$ to denote a
vector with all coordinates equal to $1$.

\begin{lemma}\label{lemma:Q}
  Let $\bfQ=(P_1/T, \ldots, P_m/T)$ be the $m$-dimensional random variable
  where the $P_j$ are independent random variables with the same
  distribution as $P$
  in Lemma~\ref{lemma:P} and $T$ is the fixed constant defined in Theorem~\ref{theorem:Berry-Esseen-dimension-m}.

  Then $\bfQ$ has density function and characteristic function
  \begin{align*}
    f_\bfQ(\bfz)&=\prod_{j=1}^m Tf_P(Tz_j),\\
    \varphi_\bfQ(\bft)&=\prod_{j=1}^m\varphi_P\Bigl(\frac{t_j}{T}\Bigr),
  \end{align*}
  respectively. The characteristic function vanishes outside $[-T, T]^m$.

  Furthermore,
  \begin{align}
    \int_{\bfz\in\R^m} \abs{z_j} f_\bfQ\Bigl(\bfz + \frac{\theta
      \lambda}{T}\bfone\Bigr)\,d\bfz &\le \frac{C_2+\lambda}{T},\label{eq:Q-integral-1}\\
    \int_{\theta \bfz\le 0} f_\bfQ\Bigl(\bfz + \frac{\theta
      \lambda}{T} \bfone\Bigr)\,d\bfz=\frac{3}{4}\label{eq:Q-integral-2}
  \end{align}
  hold for $\theta\in\{\pm 1\}$ and $j\in\{1,\ldots, m\}$.
\end{lemma}
\begin{proof}
  Because of independence, the distribution function and
  the characteristic function of $\bfQ$ is the product of the distribution
  functions and the characteristic functions of the $P_j/T$, respectively. Division by $T$ transforms the density and
  characteristic functions as claimed. As $\varphi_P(t)$ vanishes outside $[-1,
  1]$ by \eqref{eq:characteristic-function-P}, $\varphi_\bfQ(\bft)$ vanishes
  outside $[-T, T]^m$.

  By a simple translation, the integral on the left hand side of
  \eqref{eq:Q-integral-1} can be seen to be equal to
  \begin{equation*}
    \E\Bigl(\abs[\Big]{Q_j-\frac{\theta\lambda}{T}}\Bigr).
  \end{equation*}
  Then \eqref{eq:Q-integral-1} is a simple consequence of $Q_j=P_j/T$,
  \eqref{eq:expectation-P-absolute-value} and the triangle inequality.

  By the same translation and the definition of $\lambda$, the integral on the left hand side
  of~\eqref{eq:Q-integral-2} is
  \begin{equation*}
    \P\Bigl(\theta\bfQ\le \frac{\lambda}{T}\bfone\Bigr)=\prod_{j=1}^m \P(\theta
    P_j\le \lambda)=\frac{3}{4}.
  \end{equation*}
\end{proof}

From now on, we let $\bfQ$ be as in Lemma~\ref{lemma:Q} and let $\bfQ$ be
independent of $\bfX$ and independent of $\bfY$. We first prove an inequality
relating the difference between the distribution functions of $\bfX$ and $\bfY$
to that of the distribution functions of $\bfX+\bfQ$ and $\bfY+\bfQ$.

\begin{lemma}\label{lemma:perturbation}We have
  \begin{equation}\label{eq:perturbation}
    \begin{aligned}
      \sup_{\bfz\in\R^m}\abs{F_{\bfX+\bfQ}(\bfz)-F_{\bfY+\bfQ}(\bfz)}&\le
      \sup_{\bfz\in\R^m}\abs{F_\bfX(\bfz)-F_{\bfY}(\bfz)}\\
      &\le
      2\sup_{\bfz\in\R^m}\abs{F_{\bfX+\bfQ}(\bfz)-F_{\bfY+\bfQ}(\bfz)}+\frac{2\sum_{j=1}^m
        A_j}{T}(C_1+C_2).
    \end{aligned}
  \end{equation}
\end{lemma}
\begin{proof}
  Let
  \begin{align*}
    S&=\sup_{\bfz\in\R^m}\abs{F_\bfX(\bfz)-F_{\bfY}(\bfz)}\\
    S'&=\sup_{\bfz\in\R^m}\abs{F_{\bfX+\bfQ}(\bfz)-F_{\bfY+\bfQ}(\bfz)}
  \end{align*}
  and $\varepsilon>0$. We choose $\theta\in\{\pm 1\}$ such that $S=\sup_{z\in\R^m}\theta(F_{\bfX}(\bfz)-F_{\bfY}(\bfz))$.

  There is a $\bfz_{\varepsilon}\in\R^m$ such that
  \begin{equation*}
    S-\varepsilon \le \theta(F_{\bfX}-F_{\bfY})(\bfz_\varepsilon).
  \end{equation*}
  Let $\bfw\in\R^n$ with $\theta\bfw\le \bfzero$. By monotonicity of
  $F_{\bfX}$, we have $\theta F_{\bfX}(\bfz_\varepsilon-\bfw)\ge \theta
  F_{\bfX}(\bfz_\varepsilon)$. Thus
  \begin{align*}
    \theta(F_{\bfX}-F_{\bfY})(\bfz_\varepsilon-\bfw)&\ge\theta
    (F_{\bfX}-F_{\bfY})(\bfz_{\varepsilon})-\theta(F_{\bfY}(\bfz_\varepsilon-\bfw)-F_{\bfY}(\bfz_\varepsilon))\\
    &\ge S-\varepsilon - \sum_{j=1}^m A_j \abs{w_j}.
  \end{align*}
  We multiply this inequality by
  $f_{\bfQ}\bigl(\bfw+\frac{\theta\lambda}{T}\bfone\bigr)$ and integrate over all
  $\bfw\in\R^n$ with $\theta\bfw\le \bfzero$. By \eqref{eq:Q-integral-2} and
  \eqref{eq:Q-integral-1}, we get
  \begin{equation}\label{eq:integral-I_1}
    I_1:=\int_{\theta\bfw\le\bfzero}\theta(F_{\bfX}-F_{\bfY})(\bfz_\varepsilon-\bfw)f_{\bfQ}\Bigl(\bfw+\frac{\theta\lambda}{T}\bfone\Bigr)
    \,d\bfw \ge \frac{3}{4}(S-\varepsilon)-\frac{C_2+\lambda}{T}\sum_{j=1}^m A_j.
  \end{equation}
  Setting
  \begin{equation*}
    I_2:=\int_{\theta\bfw\nleq\bfzero}\theta(F_{\bfX}-F_{\bfY})(\bfz_\varepsilon-\bfw)f_{\bfQ}\Bigl(\bfw+\frac{\theta\lambda}{T}\bfone\Bigr)
    \,d\bfw
  \end{equation*}
  and using the estimate
  $\abs{\theta(F_{\bfX}-F_{\bfY})(\bfz_\varepsilon-\bfw)}\le S$ yields
  \begin{equation}\label{eq:integral-I_2}
    \abs{I_2}\le S \int_{\theta\bfw\nleq\bfzero}f_{\bfQ}\Bigl(\bfw+\frac{\theta\lambda}{T}\bfone\Bigr)
    \,d\bfw=\frac{S}{4}
  \end{equation}
  by \eqref{eq:Q-integral-2} and the fact that $f_{\bfQ}$ is a probability density function.

  Combining \eqref{eq:integral-I_1} and \eqref{eq:integral-I_2} yields
  \begin{equation}\label{eq:I_1+I_2-bound}
    \abs{I_1+I_2}\ge \abs{I_1}-\abs{I_2}\ge I_1-\abs{I_2}\ge \frac{S}{2}-\frac{C_2+\lambda}{T}\sum_{j=1}^m A_j-\frac{3\varepsilon}{4}.
  \end{equation}
  As the sum of random variables corresponds to a convolution, we have
  \begin{equation}\label{eq:convolution}
    (F_{\bfX+\bfQ}-F_{\bfY+\bfQ})(\bfz) =
    \int_{\R^m}(F_{\bfX}-F_{\bfY})(\bfz-\bfw) f_{\bfQ}(\bfw)\,d\bfw.
  \end{equation}
  Replacing $\bfz$ and $\bfw$ by $\bfz_{\varepsilon}+\frac{\theta\lambda}{T}\bfone$ and
  $\bfw+\frac{\theta\lambda}{T}\bfone$, respectively, and using
  \eqref{eq:I_1+I_2-bound} leads to
  \begin{equation*}
    S'\ge\abs[\Big]{(F_{\bfX+\bfQ}-F_{\bfY+\bfQ})\Bigl(\bfz_\varepsilon+\frac{\theta\lambda}{T}\bfone\Bigr)}=\abs{I_1+I_2}\ge\frac{S}{2}-\frac{C_2+\lambda}{T}\sum_{j=1}^m A_j-\frac{3\varepsilon}{4}
  \end{equation*}
  for all $\varepsilon>0$. Taking the
  limit for $\varepsilon\to 0$ and rearranging yields the right hand side of
  \eqref{eq:perturbation}.

  The left hand side of \eqref{eq:perturbation} is an immediate consequence of
  \eqref{eq:convolution}.
\end{proof}

We are now able to bound the difference of the distribution functions by their
characteristic functions. 

\begin{lemma}\label{le:distribution-sum-bound}We have
  \begin{multline}\label{eq:distribution-sum-bound}
    \sup_{\bfz\in\R^m}\abs[\bigg]{\sum_{\alpha\in\Pi_L}\mu_\alpha\biggl(\prod_{J\in\alpha}
    F_{\bfX_J+\bfQ_J}-\prod_{J\in\alpha}F_{\bfY_J+\bfQ_J}\biggr)(\bfz)}\\\le
  \frac1{(2\pi)^m}\int_{\norm{\bft}\le T}\abs[\bigg]{\frac{\Lambda(\varphi_{\bfX})(\bft)-\Lambda(\varphi_{\bfY})(\bft)}{\prod_{\ell\in
    L} t_\ell}}\,d\bft.
  \end{multline}
\end{lemma}
\begin{proof}
  Let $\bfa$, $\bfz\in\R^m$ with $\bfa\le \bfz$.

  The random variable $\bfX_J+\bfQ_J$ admits a
  density function, because $\bfQ_J$ admits a density function. In particular,
  $\bfX_J+\bfQ_J$ is a continuous random variable. By L\'evy's theorem
  (see, e.g., \cite[Thm.~1.8.4]{Ushakov:1999:selec-topic-charac-funct}),
  \begin{equation*}
    \P(\bfa_J \le \bfX_J+\bfQ_J\le
    \bfz_J)=\frac1{(2\pi)^{\abs{J}}}\lim_{\substack{T_j\to\infty\\j\in
        J}}\int_{\substack{-T_j\le t_j\le T_j\\ j\in J}} \varphi_{\bfX_J+\bfQ_J}(\bft_J)\prod_{j\in
      J}\frac{e^{-it_jz_j}-e^{-it_ja_j}}{-it_j}\, d\bft_J.
  \end{equation*}
  As
  $\varphi_{\bfX_J+\bfQ_J}(\bft_J)=\varphi_{\bfX_J}(\bft_J)\varphi_{\bfQ_J}(\bft_J)$
  and $\varphi_{\bfQ_J}(\bft_J)$ vanishes outside $[-T, T]^J$ by
  Lemma~\ref{lemma:Q}, we can replace the limit $T_j\to\infty$ by setting
  $T_j=T$, i.e.,
  \begin{equation*}
    \P(\bfa_J \le \bfX_J+\bfQ_J\le
    \bfz_J)=\frac{i^{\abs{J}}}{(2\pi)^{\abs{J}}}\int_{\norm{\bft_J}\le T} \varphi_{\bfX_J}(\bft_J)\varphi_{\bfQ_J}(\bft_J)\prod_{j\in
      J}\frac{e^{-it_jz_j}-e^{-it_ja_j}}{t_j}\, d\bft_J.
  \end{equation*}
  Taking the product over all $J\in\alpha$ and summing over $\alpha\in\Pi_L$
  yields
  \begin{multline}\label{eq:Levy-long-1}
    \sum_{\alpha\in\Pi_L}\mu_\alpha\prod_{J\in\alpha}\P(\bfa_J \le \bfX_J+\bfQ_J\le
    \bfz_J) \\
    = \frac{i^m}{(2\pi)^m}\int_{\norm{\bft}\le T} \varphi_{\bfQ}(\bft) \prod_{\ell\in
      L}\frac{e^{-it_\ell z_\ell}-e^{-it_\ell a_\ell}}{t_\ell}\sum_{\alpha\in
      \Pi_L}\mu_\alpha\prod_{J\in\alpha}\varphi_{\bfX_J}(\bft_J) \, d\bft
  \end{multline}
  where Fubini's theorem and the fact that
  $\varphi_\bfQ(\bft)=\prod_{J\in\alpha}\varphi_{\bfQ_J}(\bft_J)$ have been
  used. By definition of $\varphi_{\bfX}$, we have
  $\varphi_{\bfX_J}(\bft_J)=\varphi_{\bfX}(\psi_{J, L}(\bft))$. Therefore, we
  can use the definition of $\Lambda(\varphi_\bfX)$ to rewrite
  \eqref{eq:Levy-long-1} to
  \begin{equation*}
    \sum_{\alpha\in\Pi_L}\mu_\alpha\prod_{J\in\alpha}\P(\bfa_J \le \bfX_J+\bfQ_J\le
    \bfz_J)
    = \frac{i^m}{(2\pi)^m}\int_{\norm{\bft}\le T} \frac{\Lambda(\varphi_\bfX)(\bft)}{\prod_{\ell\in L}t_\ell}\varphi_{\bfQ}(\bft) \prod_{\ell\in
      L}(e^{-it_\ell z_\ell}-e^{-it_\ell
        a_\ell}) \, d\bft.
  \end{equation*}
  This equation remains valid when replacing $\bfX$ by $\bfY$; taking the
  difference results in
  \begin{multline}\label{eq:Levy-long-X-Y}
    \sum_{\alpha\in\Pi_L}\mu_\alpha\biggl(\prod_{J\in\alpha}\P(\bfa_J \le \bfX_J+\bfQ_J\le
    \bfz_J)-
\prod_{J\in\alpha}\P(\bfa_J \le \bfY_J+\bfQ_J\le
    \bfz_J)\biggr)\\
    = \frac{i^m}{(2\pi)^m}\int_{\norm{\bft}\le T} \frac{\Lambda(\varphi_\bfX)(\bft)-\Lambda(\varphi_\bfY)(\bft)}{\prod_{\ell\in L}t_\ell}\varphi_{\bfQ}(\bft) \prod_{\ell\in
      L}(e^{-it_\ell z_\ell}-e^{-it_\ell
        a_\ell}) \, d\bft.
  \end{multline}
  If the integral on the right hand side of \eqref{eq:distribution-sum-bound}
  is infinite, there is nothing to show. Thus we may assume that it is finite.
  This also implies that
  \begin{equation*}
    \frac{\Lambda(\varphi_\bfX)(\bft)-\Lambda(\varphi_\bfY)(\bft)}{\prod_{\ell\in L}t_\ell}\varphi_{\bfQ}(\bft)
  \end{equation*}
  is an integrable function on $\R^m$ (as it vanishes outside $[-T,
  T]^m$). Then by the Riemann--Lebesgue lemma,
  we may take the limit $a_\ell\to-\infty$ for all $\ell\in L$ in
  \eqref{eq:Levy-long-X-Y} to obtain
  \begin{multline*}
    \sum_{\alpha\in\Pi_L}\mu_\alpha\biggl(\prod_{J\in\alpha}\P(\bfX_J+\bfQ_J\le
    \bfz_J)-
\prod_{J\in\alpha}\P(\bfY_J+\bfQ_J\le
    \bfz_J)\biggr)\\
    = \frac{i^m}{(2\pi)^m}\int_{\norm{\bft}\le T} \frac{\Lambda(\varphi_\bfX)(\bft)-\Lambda(\varphi_\bfY)(\bft)}{\prod_{\ell\in L}t_\ell}\varphi_{\bfQ}(\bft) e^{-i\ip{\bft}{\bfz}} \, d\bft.
  \end{multline*}
  Taking absolute values and rewriting the left hand side in terms of marginal
  distribution functions yields \eqref{eq:distribution-sum-bound}.
\end{proof}

We now bound the contribution of the lower dimensional distributions.

\begin{lemma}\label{lemma:lower-dimension}
  We have
  \begin{equation*}
    \sup_{\bfz\in\R^m}\abs[\bigg]{\sum_{\substack{\alpha\in\Pi_L\\\alpha\neq \{L\}}}\mu_\alpha\biggl(\prod_{J\in\alpha}
    F_{\bfX_J+\bfQ_J}-\prod_{J\in\alpha}F_{\bfY_J+\bfQ_J}\biggr)(\bfz)}
  \le
  \sum_{\emptyset\neq J\subsetneq
    L}B_{m-\abs{J}}\sup_{\bfz\in\R^J}\abs[\big]{F_{\bfX_{J}}(\bfz)-F_{\bfY_{J}}(\bfz)}
  .
  \end{equation*}
\end{lemma}
\begin{proof}
  Let $\alpha=\{J_1,\ldots, J_r\}\in\Pi_L$. Then
  \begin{align*}
    \abs[\bigg]{\prod_{J\in\alpha} F_{\bfX_J+\bfQ_J}(\bfz_J) - \prod_{J\in\alpha}
    F_{\bfY_J+\bfQ_J}(\bfz_J)} &=
    \biggl\lvert\sum_{k=1}^r \biggl(\prod_{j=1}^k F_{\bfX_{J_j}+\bfQ_{J_j}}(\bfz_{J_j})\prod_{j={k+1}}^r
    F_{\bfY_{J_j}+\bfQ_{J_j}}(\bfz_{J_j}) \\
    &\qquad\qquad- \prod_{j=1}^{k-1} F_{\bfX_{J_j}+\bfQ_{J_j}}(\bfz_{J_j})\prod_{j={k}}^r
    F_{\bfY_{J_j}+\bfQ_{J_j}}(\bfz_{J_j})\biggr)\biggr\rvert\\
    &=
    \biggl\lvert\sum_{k=1}^r \prod_{j=1}^{k-1} F_{\bfX_{J_j}+\bfQ_{J_j}}(\bfz_{J_j})\prod_{j={k+1}}^r
    F_{\bfY_{J_j}+\bfQ_{J_j}}(\bfz_{J_j}) \\
    &\qquad\qquad\times
    \bigl(F_{\bfX_{J_k}+\bfQ_{J_k}}(\bfz_{J_j})-F_{\bfY_{J_k}+\bfQ_{J_k}}(\bfz_{J_j})\bigr)\biggr\rvert\\
    &\le \sum_{J\in\alpha}\abs[\big]{F_{\bfX_{J}+\bfQ_{J}}(\bfz_J)-F_{\bfY_{J}+\bfQ_{J}}(\bfz_J)}
  \end{align*}
  because the products over the distribution functions are bounded by $1$.

  Therefore,
  \begin{equation*}
    \abs[\bigg]{\sum_{\substack{\alpha\in\Pi_L\\\alpha\neq \{L\}}}\mu_\alpha\biggl(\prod_{J\in\alpha}
    F_{\bfX_J+\bfQ_J}-\prod_{J\in\alpha}F_{\bfY_J+\bfQ_J}\biggr)(\bfz)}
  \le
  \sum_{\emptyset\neq J\subsetneq
    L}\abs[\big]{F_{\bfX_{J}+\bfQ_{J}}(\bfz_J)-F_{\bfY_{J}+\bfQ_{J}}(\bfz_J)}
  \sum_{\substack{\alpha\in \Pi_{L}\\J\in\alpha}}\abs{\mu_\alpha}.
  \end{equation*}
  A partition $\alpha\in\Pi_L$ with $J\in \alpha$ can be uniquely written as
  $\alpha=\{J\}\cup \beta$ for a $\beta\in\Pi_{L\setminus J}$. Thus
  \begin{equation*}
    \sum_{\substack{\alpha\in
        \Pi_{L}\\J\in\alpha}}\abs{\mu_\alpha}=\sum_{\beta\in\Pi_{L\setminus J}}
    \abs{\beta}!=\sum_{k=1}^{m-\abs{J}}
    \stirlingpartition{m-\abs{J}}{k}k!=B_{m-\abs{J}}
  \end{equation*}
  because there are
  $\stirlingpartition{m-\abs{J}}{k}$ partitions of  $L\setminus J$ with $k$ parts.
  Using the left hand side of \eqref{eq:perturbation} yields the assertion
  (more precisely, of a version of the left hand side of
  \eqref{eq:perturbation} for marginal distributions).
\end{proof}

Now, we can complete the proof of the theorem.

\begin{proof}[Proof of Theorem~\ref{theorem:Berry-Esseen-dimension-m}]
  The estimate \eqref{eq:Berry-Esseen} follows from
  Lemma~\ref{lemma:perturbation} (more precisely, the right hand side of
  \eqref{eq:perturbation}), Lemma~\ref{le:distribution-sum-bound} and
  Lemma~\ref{lemma:lower-dimension}.

  If the expectation of $\bfX$ exists, $\varphi_{\bfX}$ is
  differentiable. Therefore, $\Lambda(\varphi_\bfX)$ is differentiable,
  too. By Lemma~\ref{lemma:Lambda-property}, $\Lambda(\varphi_\bfX)(\bft)$
  has a zero whenever one of the $t_\ell$, $\ell\in L$, vanishes. Thus
  \begin{equation*}
    \frac{\Lambda(\varphi_\bfX)(\bft)}{\prod_{\ell\in L}t_\ell}
  \end{equation*}
  is bounded around $\bfzero$ and therefore bounded on $[-T, T]^m$. The same
  holds for $\bfY$. Thus the integral on the right hand side of
  \eqref{eq:Berry-Esseen} converges.
\end{proof}

\section{Proof of the Quasi-Power Theorem}\label{sec:proof-quasi-power-theorem}
We may now prove the $m$-dimensional quasi-power theorem, Theorem~\ref{th:quasi-power-dD}.

Let $\bfmu_n=\phi_{n}\grad u(\bfzero) $ and $\Sigma=H_u(\bfzero)$. We define
the random vector $\bfX=\phi_{n}^{-1/2}(\bfOmega_n-\bfmu_n)$. For
simplicity, we ignore the dependence on $n$ in this and the following notations.

First, we establish bounds for the characteristic function of $\bfX$.

\begin{lemma}\label{lemma:characteristic-function-single-bound}
 For $\Sigma$ regular or singular, there exists an analytic function $V(\bfs)$
  which is analytic for  $\norm{\bfs}< \tau\sphi/2$ such that
  \begin{equation*}
    \varphi_{\bfX}(\bfs)=\exp\Bigl(-\frac12 \bfs^\top \Sigma \bfs+V(\bfs)\Bigr)
  \end{equation*}
  and
  \begin{equation}\label{eq:V-bound}
    V(\bfs)=O\Bigl(\frac{\norm{\bfs}^3+\norm{\bfs}}{\sphi}\Bigr)
  \end{equation}
  hold for all $\bfs\in\C^K$ with $\norm{\bfs}< \tau\sphi/2$.

  For $n\to\infty$, $\bfX$ converges in distribution to a normal distribution with mean $\bfzero$
  and variance-covariance matrix $\Sigma$. In particular, $\Sigma$ is positive (semi-)definite if it is regular
  (singular, respectively).
\end{lemma}
\begin{proof}
  By replacing $u(\bfs)$ and $v(\bfs)$ by $u(\bfs)-u(\bfzero)$ and $v(\bfs)-v(\bfzero)$,
  respectively, we may assume that $u(\bfzero)=v(\bfzero)=0$.
  We define $E(\bfs)$ by the relation $M_n(\bfs)=e^{W_n(\bfs)}(1+E(\bfs))$
  and note that by assumption, $E(\bfs)=O(\kappa_n^{-1})$ uniformly for
  $\|\bfs\|\le\tau$. We note that this implies
  $E(\bfzero)=0$.

  By assumption, $M_n(\bfs)$ exists for $\norm{\bfs}\le \tau$. Therefore, it is
  continuous for these $\bfs$ and, by Morera's theorem combined with
  applications of Fubini's and Cauchy's theorems, $M_n(\bfs)$ is analytic for
  $\norm{\bfs}\le\tau$. This also implies that
  $E(\bfs)$ is analytic for $\norm{\bfs}\le\tau$. By
  Cauchy's formula, we have
  \begin{equation*}
    \frac{\partial E(\bfs)}{\partial s_j}=\frac1{2\pi
      i}\oint_{\abs{\zeta_j}=\tau}\frac{E(s_1,\ldots, s_{j-1}, \zeta_j,
      s_{j+1}, \ldots, s_d)}{(\zeta_j-s_j)^2}\,d\zeta_j=O\Bigl(\frac{1}{\kappa_n}\Bigr)
  \end{equation*}
  for $\norm{\bfs}<\tau/2$. Thus
  \begin{equation*}
    E(\bfs)=\int_{[0, \bfs]}\ip{\grad E(\bft)}{d\bft}=O\Bigl(\frac{\norm{\bfs}}{\kappa_n}\Bigr)
  \end{equation*}
  for $\norm{s}<\tau/2$.

  We
  calculate that
  \begin{align*}
    \varphi_{\bfX}(\bfs)&=M_n\big(i\phi_{n}^{-1/2}\bfs\big)\exp\big(-i\phi_{n}^{-1/2}\ip{\bfmu_n}{\bfs}\big)\\
    &=\exp\Big(-\frac12 \bfs^\top \Sigma \bfs + V(\bfs)\Big)
  \end{align*}
  with
  \begin{equation*}
    V(\bfs)=u(i\phi_{n}^{-1/2}\bfs)\phi_{n}+v(i\phi_{n}^{-1/2}\bfs)-i\phi_{n}^{-1/2}\ip{\bfmu_n}{\bfs}
    + \frac12 \bfs^\top \Sigma \bfs + \log(1+E(i\phi_{n}^{-1/2}\bfs)).
  \end{equation*}

  Since $u(\bfzero)=v(\bfzero)=0$ and the first and second order
  terms of $u$ cancel out, we have
  \begin{align*}
    V(\bfs)=O\Bigl(\frac{\|\bfs\|^3+\|\bfs\|}{\sqrt{\phi_{n}}}\Bigr)
  \end{align*}
  for $\|\bfs\| <\tau\sphi/2$.

  Note that
  \begin{equation*}
    \lim_{n\to\infty} \varphi_\bfX(\bfs)=\exp\Big(-\frac12 \bfs^\top \Sigma \bfs \Big)
  \end{equation*}
  for $\bfs\in\C^m$,
  which implies that, in distribution, $\bfX$ converges  to the normal
  distribution with mean zero and variance-covariance matrix $\Sigma$. Although
  we have
  to refine our estimates for applying Theorem~\ref{theorem:Berry-Esseen-dimension-m},
  we immediately conclude that $\Sigma$ is positive (semi-)definite
  depending on whether it is regular or not.
\end{proof}

Let now $\Sigma$ be regular. By $\bfY$ we denote a
normally distributed random variable in $\R^m$ with
mean $\bfzero$ and variance-covariance matrix $\Sigma$. Its characteristic
function is
\begin{equation*}
  \varphi_\bfY(\bfs)=\exp\Big(-\frac12 \bfs^\top \Sigma \bfs \Big).
\end{equation*}
The smallest eigenvalue of $\Sigma$ is denoted by $\sigma>0$.

We are now able to bound the functions occurring in the Berry--Esseen inequality.

\begin{lemma}\label{lemma:characteristic-function-difference-bound}
  There exists a $c<\tau/2$ such that
  \begin{equation*}
    \abs{\Lambda(\varphi_\bfX)(\bfs)-\Lambda(\varphi_\bfY)(\bfs)}\le
    \exp\Bigl(-\frac{\sigma}{4}\norm{\bfs}^2 + O(\norm{\bfs})\Bigr)O\Bigl(\frac{\norm{\bfs}^3+\norm{\bfs}}{\sphi}\Bigr)
  \end{equation*}
  holds for all $\bfs\in\C^L$ with $\norm{\bfs}\le c\sphi$ and $\norm{\Im
    \bfs}\le 1$.
\end{lemma}
\begin{proof}
  Let $\alpha\in\Pi_L$. Then by
  Lemma~\ref{lemma:characteristic-function-single-bound}, we have
  \begin{multline}\label{eq:characteristic-function-multiple-bound-1}
    \abs[\bigg]{\prod_{J\in \alpha} (\varphi_\bfX\circ\psi_{J,L})(\bfs) - \prod_{J\in \alpha} (\varphi_\bfY\circ\psi_{J,L})(\bfs)}\\
    =\exp\biggl(-\frac12\Re\sum_{J\in \alpha} \psi_{J,L}(\bfs)^\top \Sigma
    \psi_{J,L}(\bfs) \biggr)
    \abs[\bigg]{\exp\biggl(\sum_{J\in\alpha} V(\psi_{J,L}(\bfs))\biggr)-1}.
  \end{multline}
  For $\bft\in\R^L$, we have $\bft^\top \Sigma \bft\ge \sigma \bft^\top\bft\ge
  \sigma \norm{\bft}^2$. For complex $w$, we have $|\exp(w)-1|\le
  |w|\exp(|w|)$.
  Splitting $\bfs$ into its real and imaginary parts in the first summand and
  using these inequalities for the first and second factor of
  \eqref{eq:characteristic-function-multiple-bound-1}, respectively, yields
  \begin{multline*}
    \abs[\bigg]{\prod_{J\in \alpha} (\varphi_\bfX\circ\psi_{J,L})(\bfs) - \prod_{J\in \alpha} (\varphi_\bfY\circ\psi_{J,L})(\bfs)}\\
    \le \exp\Bigl(-\frac{\sigma}2\norm{\bfs}^2 +
    O\Bigl(\norm{\bfs}+\frac{\norm{\bfs}^3+ \norm{\bfs}}{\sphi}\Bigr)\Bigr)
    O\Bigl(\frac{\norm{\bfs}^3+ \norm{\bfs}}{\sphi}\Bigr)
  \end{multline*}
  by~\eqref{eq:V-bound}. For sufficiently
  small $c$, we obtain
  \begin{equation*}
    \abs[\bigg]{\prod_{J\in \alpha} (\varphi_\bfX\circ\psi_{J,L})(\bfs) - \prod_{J\in \alpha} (\varphi_\bfY\circ\psi_{J,L})(\bfs)}
    \le \exp\Bigl(-\frac{\sigma}4\norm{\bfs}^2 +
    O(\norm{\bfs})\Bigr)
    O\Bigl(\frac{\norm{\bfs}^3+ \norm{\bfs}}{\sphi}\Bigr).
  \end{equation*}
  Multiplying by $\abs{\mu_\alpha}$ and summation over all $\alpha\in\Pi_L$
  concludes the proof of the lemma.
\end{proof}

The last ingredient to prove the quasi-power theorem is a bound of the
integrals occurring in the Berry--Esseen inequality.

\begin{lemma}\label{le:integral-bound}
  Let $c$ be as in
  Lemma~\ref{lemma:characteristic-function-difference-bound}.
  Then
  \begin{equation*}
    \int_{\norm{\bfs}\le c\sphi}\abs[\Big]{\frac{\Lambda(\varphi_{\bfX})(\bfs)-\Lambda(\varphi_\bfY)(\bfs)}{\prod_{\ell\in L}s_\ell}}\, d\bfs = O\Bigl(\frac{1}{\sphi}\Bigr).
  \end{equation*}
\end{lemma}
\begin{proof}
  For simplicity, set $h=\Lambda(\varphi_{\bfX})-\Lambda(\varphi_\bfY)$.
  For a partition $\{J,K\}$ of $L$, set
  \begin{equation*}
    \calS(J, K)=\{ \bfs\in\R^L\colon \abs{s_j}\le 1\text{ for }j\in J,\ 1\le
    \abs{s_k}\le c\sphi\text{ for }k\in K\}
  \end{equation*}
  and partition $\bfs$ into $(\bfs_{J},\bfs_{K})$.
  We use the notation
  \begin{equation*}
    D^{J} = \frac{\partial^{\abs{J}}}{\partial z_{j_1}\cdots \partial z_{j_{\abs{J}}}}
  \end{equation*}
  when $J=\{j_1,\ldots, j_{\abs{J}}\}$. The product of the paths from $0$ to
  $s_j$ for $j\in J$ is denoted by $[\bfzero, \bfs_{J}]$.

  By Lemma~\ref{lemma:Lambda-property}, we have
  \begin{equation}\label{eq:integral-bound-integral}
    h(\bfs)=\int_{[\bfzero,\bfs_{J}]}
    D^{J}(h(\bfz_{J}, \bfs_{K}))\,d\bfz_{J}.
  \end{equation}
  By Cauchy's integral formula, we have
  \begin{equation}\label{eq:integral-bound-2}
    D^{J}(h(\bfz_{J}, \bfs_{K})) =
    \frac1{(2\pi
      i)^{\abs{J}}}\oint_{\bfzeta_{J}}\frac{h(\bfzeta_{J},
      \bfs_{K})}{\prod_{j\in J}(\zeta_j-z_j)^2}\,d\bfzeta_{J}
  \end{equation}
  where  $\zeta_j$ is integrated over the circle of radius $1$ around $z_j$ for
  $j\in J$, thus $\norm{\Im \bfzeta_{J}}\le 1$.

  Using the estimate of Lemma~\ref{lemma:characteristic-function-difference-bound} yields
  \begin{equation}\label{eq:integral-bound-bound}
    \begin{aligned}
      \abs{h(\bfzeta_{J}, \bfs_{K})}
      &=\exp\Bigl(-\frac{\sigma}{4}\norm{(\bfzeta_{J},
        \bfs_{K})}^2+O(\norm{(\bfzeta_{J}, \bfs_{K})})\Bigr)\\
      &\qquad\times O\Bigl(\frac{\norm{(\bfzeta_{J},
          \bfs_{K})}^3+\norm{(\bfzeta_{J}, \bfs_{K})}}{\sphi}\Bigr)\\
      &=\exp\Bigl(-\frac{\sigma}4\norm{\bfs}^2+O(\norm{\bfs}+1)\Bigr) O\Bigl(\frac{\norm{\bfs}^3+1}{\sphi}\Bigr).
    \end{aligned}
  \end{equation}

  Combining \eqref{eq:integral-bound-integral}, \eqref{eq:integral-bound-2} and
  \eqref{eq:integral-bound-bound} leads to
  \begin{align*}
    &\int_{\calS(J, K)}\abs[\Big]{\frac{h(\bfs)}{\prod_{\ell\in L}s_\ell}}\, d\bfs
    \\&\qquad=O\biggl(\frac1{\sphi}\int_{\calS(J, K)}\frac{1}{\prod_{\ell\in
        L}\abs{s_\ell}}\\
    &\qquad\qquad\times\abs[\bigg]{\int_{[\bfzero, \bfs_{J}]}\exp\Bigl(-\frac{\sigma}{4}\norm{\bfs}^2+ O(\norm{\bfs}+1)\Bigr)(\norm{\bfs}^3+1)
    \, d\bfz_{J} }d\bfs\biggr).
    \\&\qquad=O\biggl(\frac1{\sphi}\int_{\calS(J, K)}\frac{1}{\prod_{\ell\in
        L}\abs{s_\ell}}\exp\Bigl(-\frac{\sigma}{4}\norm{\bfs}^2+ O(\norm{\bfs}+1)\Bigr)(\norm{\bfs}^3+1)\\
    &\qquad\qquad\times\abs[\bigg]{\int_{[\bfzero, \bfs_{J}]}
    \, d\bfz_{J} }d\bfs\biggr).
  \end{align*}
  The inner integral results in $\abs{\prod_{j\in J}s_j}$. The factors $\abs{s_k}\geq1$
  for $k\in K$ in the denominator can simply be omitted. If
  $K\neq \emptyset$, we still
  have to bound
  \begin{align*}
    &\int_{\calS(J, K)}\exp\Bigl(-\frac{\sigma}{4}\norm{\bfs}^2+
    O(\norm{\bfs}+1)\Bigr)(\norm{\bfs}^3+1)\,d\bfs\\
    &\qquad=\sum_{k\in K} \int_{\substack{\calS(J, K)\\\norm{\bfs}=\abs{s_k}}}
    \exp\Bigl(-\frac{\sigma}{4}\norm{\bfs}^2+
    O(\norm{\bfs}+1)\Bigr)(\norm{\bfs}^3+1)\,d\bfs\\
    &\qquad=\sum_{k\in K} \int_{\substack{\calS(J, K)\\\norm{\bfs}=\abs{s_k}}}
    \exp\Bigl(-\frac{\sigma}{4}\abs{s_k}^2+
    O(\abs{s_k}+1)\Bigr)(\abs{s_k}^3+1)\,d\bfs\\
    &\qquad=\sum_{k\in K} \int_{1\le \abs{s_k}\le c\sphi}
    \exp\Bigl(-\frac{\sigma}{4}\abs{s_k}^2+
    O(\abs{s_k})\Bigr)(\abs{s_k}^3+1)\\
    &\qquad\qquad\qquad\times\int_{\bfone\le\abs{\bfs_{K\setminus
          \{k\}}}\le \abs{s_k}\bfone} \int_{\abs{\bfs_J}\le \bfone} \,d\bfs_J
    d\bfs_{K\setminus \{k\}}ds_k\\
    &\qquad=2^{\abs{L}-1}\sum_{k\in K} \int_{1\le \abs{s_k}\le c\sphi}
    \exp\Bigl(-\frac{\sigma}{4}\abs{s_k}^2+
    O(\abs{s_k})\Bigr)(\abs{s_k}^3+1)\abs{s_k}^{\abs{K}-1}\,ds_k
  \end{align*}
  where the integration bounds are meant coordinate-wise.
  Then we use the fact that
  \begin{equation*}
    \int_{x\in\R} \exp\Bigl(-\frac{\sigma}{4} x^2\Bigr)\abs{x}^t\,dx
  \end{equation*}
  is finite for all constants $t\geq 0$. Thus, after completing the
  square in the argument of the exponential function, the integral over $s_k$ is
  bounded by a constant, i.e.,
  \begin{equation*}
    \int_{\calS(J, K)}\exp\Bigl(-\frac{\sigma}{4}\norm{\bfs}^2+
    O(\norm{\bfs}+1)\Bigr)(\norm{\bfs}^3+1)\,d\bfs =O(1).
  \end{equation*}
  We conclude that
  \begin{equation*}
    \int_{\calS(J, K)}\abs[\Big]{\frac{h(\bfs)}{\prod_{\ell\in L}s_\ell}}\, d\bfs=O\Bigl(\frac1{\sphi}\Bigr).
  \end{equation*}
  Summation over all partitions $\{J,K\}$ of $L$ completes the proof of the lemma.
\end{proof}

We now collect all results to prove Theorem~\ref{th:quasi-power-dD}.

\begin{proof}[Proof of Theorem~\ref{th:quasi-power-dD}]
We set $T=c\sphi$ with $c$ from
Lemma~\ref{lemma:characteristic-function-difference-bound}. By
Theorem~\ref{theorem:Berry-Esseen-dimension-m} and
Lemma~\ref{le:integral-bound}, we have
\begin{equation}\label{eq:quasi-power-one-step}
    \sup_{\bfz\in\R^m}\abs{F_{\bfX}(\bfz)-F_{\bfY}(\bfz)}=O\Bigl(\frac1{\sphi}\Bigr)+O\biggl(\sum_{\emptyset\neq
    J\subsetneq L} \sup_{\bfz_J\in\R^J}\abs[\big]{F_{\bfX_{J}}(\bfz_J)-F_{\bfY_{J}}(\bfz_J)}\biggr).
\end{equation}
For $\emptyset\neq J\subsetneq L$, we have
$\varphi_{\bfX_J}=\varphi_{\bfX}\circ \chi_{J, L}$. Therefore, all prerequisites
for applying the quasi-power theorem on $(\bfOmega_n)_J$ are fulfilled.
Therefore, we can apply \eqref{eq:quasi-power-one-step} recursively
and finally obtain
\begin{equation*}
   \sup_{\bfz\in\R^m}\abs{F_{\bfX}(\bfz)-F_{\bfY}(\bfz)}=O\Bigl(\frac1{\sphi}\Bigr).
 \end{equation*}
\end{proof}
Note that it would also have been possible to apply
Corollary~\ref{corollary:Berry-Esseen}; however, this would have required
proving Lemmas~\ref{lemma:characteristic-function-difference-bound} and~\ref{le:integral-bound} for subsets $K$ of $L$, which would have required
some notational overhead using $\chi_{K, L}$.

\begin{proof}[Proof of Proposition~\ref{proposition:moments}]
  This follows by the same arguments as in \cite[Thm.~2]{Hwang:1998}.
\end{proof}

\bibliographystyle{amsplainurl}
\bibliography{bib/cheub}
\end{document}